\documentclass[reqno, 12pt]{article}

\pdfoutput=1

\usepackage{enumerate}
\usepackage{latexsym}
\usepackage[centertags]{amsmath}
\usepackage{amsfonts}
\usepackage{amssymb}
\usepackage{amsthm}
\usepackage{mathtools}
\usepackage{newlfont}
\usepackage{graphics}
\usepackage{color}
\usepackage{float}
\usepackage{diagbox}
\usepackage{tocloft}
\usepackage{titlesec}
\usepackage{booktabs}
\usepackage{extpfeil}
\usepackage{centernot}
\usepackage[pagebackref]{hyperref}
\usepackage[linesnumbered,ruled,vlined]{algorithm2e}
\usepackage{url}
\usepackage[T1]{fontenc}
\usepackage{lmodern}
\usepackage{hyperref}
\usepackage{microtype}
\usepackage[nameinlink,noabbrev,capitalize]{cleveref}
\usepackage{longtable}
\usepackage{rotating}
\usepackage{multirow}
\usepackage{extarrows}
\usepackage[sort,compress,numbers]{natbib}
\usepackage[utf8]{inputenc}
\numberwithin{equation}{section}

\newtheorem{theorem}{Theorem}[section]
\newtheorem{proposition}[theorem]{Proposition}
\newtheorem{lemma}[theorem]{Lemma}
\newtheorem{corollary}[theorem]{Corollary}
\theoremstyle{definition}
\newtheorem{definition}[theorem]{Definition}
\newtheorem{remark}[theorem]{Remark}
\newtheorem{example}[theorem]{Example}

\newtheorem{prob}[theorem]{Open Problem}

\allowdisplaybreaks[4]

\SetKwInput{KwInput}{Input}                
\SetKwInput{KwOutput}{Output}              

\newcommand{\Aff}{\mathbb A}
\newcommand{\FF}{\mathbb F}
\newcommand{\NN}{\mathbb{N}}
\newcommand{\ZZ}{\mathbb{Z}}
\newcommand{\CC}{\mathbb{C}}

\newcommand{\cA}{\mathcal{A}}

\newcommand{\cB}{\mathcal B}
\newcommand{\cD}{\mathcal D}

\newcommand{\gen}[1]{\left\langle #1\right\rangle}

\newcommand{\cond}{\operatorname{c}}
\newcommand{\Gap}{\operatorname{Gap}}
\newcommand{\remp}{\operatorname{rem}}

\title{On the Frobenius Number of Quotients of Numerical Semigroups}

\author{Feihu Liu
\\[2mm]
{\small Center for Combinatorics, LPMC}\\[-0.8ex]
{\small Nankai University, Tianjin 300071, P.R.~China}\\
{\small Email address: liufeihu7476@163.com}\\
}

\date{\today}

\begin{document}

\maketitle

\begin{abstract}
Given a numerical semigroup $S$ and a positive integer $p$, the quotient $\frac{S}{p}=\{n\in \mathbb{N} \mid pn\in S\}$ also forms a numerical semigroup. When $S=\langle a,b\rangle$ with $\gcd(a,b)=1$, a well-known open problem is to find a closed-form formula for the Frobenius number $g\!\left(\frac{\langle a,b\rangle}{p}\right)$, which remains open even in the special case $b=a+1$. Inspired by Curtis's theorem on the non-existence of polynomial formulas for the Frobenius number $g(\langle s_1,s_2,s_3\rangle)$, we provide a negative answer to this open problem in a certain sense. Concretely, we obtain the following three main results.

(i): The Frobenius number $g\!\left(\frac{\langle a,b\rangle}{p}\right)$ cannot be represented, uniformly in $a,b,p$, by any finite collection of polynomial (or rational) formulas.

(ii): For each fixed $p$, the function $a\mapsto g\!\left(\frac{\langle a,a+1\rangle}{p}\right)$ is a quadratic quasi-polynomial with period dividing $p$.

(iii): There is no nonzero polynomial $F\in \mathbb{C}[X_1,X_2,X_3]$ satisfying $F\left(a,p,g\!\left(\frac{\langle a,a+1\rangle}{p}\right)\right)=0$ for all primes $a,p$ with $2<p<a$; the same conclusion already holds if only $p$ is required to be prime and $a$ ranges over all integers greater than $p$.

While (iii) is stronger than (i), the proofs of the two results reveal different insights. Dirichlet's theorem on primes in arithmetic progressions plays a crucial role in our arguments.
\end{abstract}

\noindent
\begin{small}
\end{small}

\noindent
\begin{small}
\end{small}

\tableofcontents

\section{Introduction}

Throughout the paper, $\mathbb{N}$ and $\ZZ_{>0}$ denote the sets of all non-negative integers and all positive integers, respectively.

A subset $S$ of $\mathbb{N}$ is called a \emph{submonoid} if it contains $0$ and is closed under addition in $\mathbb{N}$.
When $\mathbb{N}\setminus S$ is finite, we call $S$ a \emph{numerical semigroup}.
Given a sequence $A = (a_1, a_2, \ldots, a_k)$ of positive integers, we denote by $\langle A \rangle$ the submonoid of $\mathbb{N}$ generated by $A$, which consists of all finite linear combinations of elements of $A$ with non-negative integer coefficients:
$$\langle A \rangle = \left\{ x_1a_1+x_2a_2+\cdots+x_ka_k \mid x_i \in \mathbb{N}\quad \text{for}\quad 1 \leq i \leq k \right\}.$$
In \cite[Lemma 2.1]{J.C.Rosales}, it is shown that $\langle A\rangle$ forms a numerical semigroup if and only if $\gcd(A)=1$.
When this condition holds, we refer to $A$ as a \emph{system of generators} for the numerical semigroup $\langle A\rangle$.
For more descriptions and results about numerical semigroups, see \cite{Ramrez Alfonsn,A.Assi,J.C.Rosales}.

Let $S\neq\NN$ be a numerical semigroup.  Its set of \emph{gaps} is $\Gap(S):=\NN\setminus S$.
Since $\Gap(S)$ is finite and nonempty, its largest element exists.  The \emph{Frobenius number} of $S$ is $g(S):=\max\Gap(S)$.
The \emph{conductor} of $S$ is $\cond(S):=g(S)+1$.
Equivalently, $\cond(S)$ is the least integer $c$ such that every integer $n\ge c$ belongs to $S$; $g(S)$ is the greatest integer not belonging to $S$.
The \emph{genus} (or \emph{Sylvester number}) of $S$ is $\# \Gap(S)$, which is the cardinality of $\Gap(S)$.

Given a numerical semigroup $S$, computing $g(S)$ and $\Gap(S)$ is often referred to as the \emph{Frobenius Problem} or the \emph{Coin Exchange Problem}, see \cite{Ramrez Alfonsn}. The determination of $g(S)$ was shown by Ram\'irez Alfons\'in \cite{RamrezAlfonsn-Combin} to be NP-hard under Turing reduction.
This naturally attracts many scholars to study this invariant.
For example, see \cite{J. E. Shockley,LiuSemigroupF,Rosales2015,Rosales.Repunit,Rosales2016,E. S. Selmer}.

When $\gcd(a,b)=1$, Sylvester \cite{J.J.Sylvester,J.J.Sylvester1} established the following formulas 
\begin{align*}
g(\langle a,b\rangle) = ab - a - b \quad \text{and} \quad \# \Gap(\langle a,b\rangle) = \frac{1}{2}(a-1)(b-1).
\end{align*}
For the case $S=\langle a,b,c\rangle$ with $\gcd(a,b,c)=1$, Curtis \cite{F.Curtis} demonstrated that no closed-form expression of a certain type exists for $g(A)$. Nevertheless, numerous special cases have been investigated; we refer to \cite[Section 2; Section 3]{Ramrez Alfonsn} and \cite{Liu-Xin} for detailed results.

Let $S$ be a numerical semigroup and let $p\in\ZZ_{>0}$.  The \emph{quotient of $S$ by $p$} is
\[\frac{S}{p}:=\{n\in\NN:pn\in S\}.\]
This definition was introduced in \cite{Rosales03}.
It is easy to verify (see \cite[Section 5]{J.C.Rosales}) that $\frac{S}{p}$ is a numerical semigroup, that $S \subseteq\frac{S}{p}$, and that $\frac{S}{p}=\mathbb{N}$ if and only if $p\in S$. 
Although this quotient is again a numerical semigroup, its generators may have many elements, see \cite{Rosales05,Rosales2006,E.Cabanillas,Alessio19,Liu-BAMS}.

When $S=\langle a,b\rangle$ with $\gcd(a,b)=1$, a natural problem is to determine explicit formulas for the Frobenius number $g\left(\frac{\langle a,b\rangle}{p}\right)$ and the genus $\# \Gap\left(\frac{\langle a,b\rangle}{p}\right)$.
If $p$ is a positive divisor of $a$, then $\frac{\langle a,b\rangle}{p}=\langle \frac{a}{p},b\rangle$ (see \cite{Rosales05}). Therefore, we have
\begin{align*}
g\left(\frac{\langle a,b\rangle}{p}\right)=\frac{ab}{p}-\frac{a}{p}-b\quad \text{and}\quad \# \Gap\left(\frac{\langle a,b\rangle}{p}\right)=\frac{1}{2p}(a-p)(b-1)\quad \text{for}\quad p\mid a.
\end{align*}

However, for general $p$, this problem remains open and has been mentioned repeatedly in the literature. 
In this paper, we mainly focus on Frobenius number $g\left(\frac{\langle a,b\rangle}{p}\right)$.
\begin{prob}{(\cite[Open Problem 6.20]{J.C.Rosales}, \cite[Section 3]{MDelgado13}, \cite[Section 5.4]{E.Cabanillas})}\label{Openproblem}
Let $a$ and $b$ be two positive relatively prime integers and let $p\in\ZZ_{>0}$.
Find a closed-form formula for $g\left(\frac{\langle a,b\rangle}{p}\right)$.
Furthermore, what happens when $b=a+1$?
\end{prob}

For the case $p=2$ and the case $p=5$, $b=a+1$, the closed formulas for the Frobenius number $g\left(\frac{\langle a,b\rangle}{p}\right)$ is derived in \cite{F.Strazzanti}.
When the generators of $S$ form an arithmetic progression, i.e., $S=\langle a,a+d,\ldots,a+kd\rangle$ ($\gcd(a,d)=1$), and $p$ is a positive divisor of $a$, the corresponding formula for the Frobenius number $g\left(\frac{S}{p}\right)$ is studied in \cite{A.Adeniran} and \cite{Liu-IJAC}. Recently, some related developments can be found in \cite{Liu-GCNSNS}.

The motivation of this paper comes from the paper \cite{F.Curtis} in which Curtis proved that the $g(\langle a,b,c\rangle)$ can not be given by closed formulas of a certain type.
Inspired by Curtis's work and prompted by Open Problem \ref{Openproblem}, we also establish several polynomial nonexistence theorems concerning formulas of $g\left(\frac{\langle a,b\rangle}{p}\right)$.

Define
\begin{align*}
  \cA:=\bigl\{(a,b,p)\in\ZZ_{>0}^{3}: \ 2<p<a<b,\quad p\text{ and }a\text{ are prime},\quad  \gcd(a,b)=1,\quad p\nmid b\bigr\}.
\end{align*}

The \textbf{first} main contribution of this paper is as follows.
\begin{theorem}\label{thm:main}
There is no nonzero polynomial $F\in\CC[X_1,X_2,X_3,Y]$ such that
\begin{equation}\label{eq:main-relation}
  F\left(a,b,p,
  g\!\left(\frac{\gen{a,b}}{p}\right)\right)=0
\end{equation}
for every $(a,b,p)\in\cA$.
\end{theorem}

The theorem rules out not merely a single polynomial expression for the Frobenius number, but every algebraic relation of bounded degree with coefficients independent of the parameters.  Its immediate corollary is the following.

\begin{corollary}\label{cor:finite-poly-intro}
There do not exist finitely many polynomials $f_1,\ldots,f_m\in\CC[X_1,X_2,X_3]$ with the property that, for every $(a,b,p)\in\cA$, at least one index $i\in\{1,\ldots,m\}$ satisfies
\[f_i(a,b,p)=g\!\left(\frac{\gen{a,b}}{p}\right).
\]
In particular, no such finite family can work for all relatively prime positive integers $a<b$ and all positive integers $p<a$.
\end{corollary}

At this point, a natural idea is to consider the formula for the Frobenius number $g\!\left(\frac{\gen{a,a+1}}{p}\right)$ under the stronger condition $b=a+1$.
For convenience, we set 
$$S_a=\gen{a,a+1}\quad \text{and} \quad G(a,p)=g\left(\frac{S_a}{p}\right).$$
Since $\gcd(a,a+1)=1$, this is a numerical semigroup for every $a\ge2$.
The question is whether the restriction $b=a+1$ makes the Frobenius number algebraically simpler.

For an integer $u$, let $\remp_p(u)$ denote the unique integer in
$\{0,1,\ldots,p-1\}$ congruent to $u$ modulo $p$.  For
$r\in\{0,1,\ldots,p-1\}$ and $0\le i\le p-1$, define
\begin{align*}
\delta_{p,r}(i) :=\remp_p\bigl(r^2-r-1-ir\bigr).
\end{align*}
We shall prove that the following minimum is always defined:
\begin{align*}
   I_p(r)&:=\min\{i\in\{0,\ldots,p-1\}:\delta_{p,r}(i)\le i\},\\
   J_p(r)&:=\delta_{p,r}\bigl(I_p(r)\bigr).
\end{align*}

Our \textbf{second} main result gives an exact formula on every residue class modulo $p$.  

\begin{theorem}\label{thm:intro-exact}
Let $a$ and $p$ be integers satisfying $a>p\ge2$, and let
$r=\remp_p(a)$.  Then
\begin{align*}
  G(a,p)=\frac{a^2-(I_p(r)+1)a-(J_p(r)+1)}{p}.
\end{align*}
In particular, for every fixed $p$, the function $a\mapsto G(a,p)$ on
$a>p$ is a quadratic quasi-polynomial with period dividing $p$.
\end{theorem}

Although there are at most $p$ residue-class branches, their number is not bounded independently of $p$.
Thus, we can still derive a theorem on the non-existence of a polynomial formula for $G(a,p)$.
Define
\begin{align}
 \cD_{\mathrm{all}}&:=\{(a,p)\in\ZZ_{>0}^2:p\text{ is prime and }p<a\},\label{eq:Dall}\\
 \cD_{\mathrm{pp}}&:=\{(a,p)\in\ZZ_{>0}^2:a,p\text{ are prime and }2<p<a\}.\label{eq:Dpp}
\end{align}

Our \textbf{third} main result is as follows. Although Theorem \ref{thm:intro-reduced} is stronger than Theorem \ref{thm:main}, their proofs reflect different ideas.

\begin{theorem}\label{thm:intro-reduced}
There is no nonzero polynomial $F\in\CC[X_1,X_2,X_3]$ such that $F\bigl(a,p,G(a,p)\bigr)=0$ for every $(a,p)\in\cD_{\mathrm{all}}$. 
In fact, the same conclusion holds with $\cD_{\mathrm{all}}$ replaced by the thinner set $\cD_{\mathrm{pp}}$.
\end{theorem}

The assertion for $\cD_{\mathrm{all}}$ is elementary once Theorem~\ref{thm:intro-exact} is established.
The stronger assertion for $\cD_{\mathrm{pp}}$ uses Dirichlet's theorem on primes in arithmetic progressions.

\begin{corollary}\label{Corollary-II-PP}
There do not exist finitely many polynomials $f_1,\ldots,f_m\in\CC[X_1,X_2]$ such that, for every
$(a,p)\in\cD_{\mathrm{pp}}$, one has $G(a,p)=f_i(a,p)$ for at least one index $i$.  The same statement holds for finite collections
of rational functions in $\CC(X_1,X_2)$, provided the selected rational function is defined at the point under consideration.
Consequently, no finite collection of polynomial or rational formulas can describe $G(a,p)$ for all integers $a>p\ge2$.
\end{corollary}

The paper is organized as follows. Section~\ref{Section-Two-Elementary} gives a complete description of membership and
gaps in $\gen{a,b}$ and $\gen{a,a+1}$. We also present some basic properties and results that will be needed later.
Section~\ref{Section-Three-FN} derives a closed-form formula for the Frobenius number of the quotient of a certain class of numerical semigroups.
In Section~\ref{Section-Four-TheoI}, we mainly use Dirichlet's theorem to prove Theorem \ref{thm:main} and Corollary \ref{cor:finite-poly-intro}. The proof is technical.
Section~\ref{sec:exact} proves the exact residue-class formula in Theorem~\ref{thm:intro-exact}.
Section~\ref{sec:global} proves Theorem \ref{thm:intro-reduced} and Corollary \ref{Corollary-II-PP}, and discusses the Zariski density. Section~\ref{Section-Finily-CR} contains the concluding remarks.

\section{Elementary properties of quotient semigroups}\label{Section-Two-Elementary}

We give a completely explicit description of membership in $S=\gen{a,b}$.

\begin{lemma}\label{lem:membership2}
Let $a,b\in\ZZ_{>0}$ satisfy $\gcd(a,b)=1$, and let $S=\gen{a,b}$.  For every $n\in\NN$, there is a unique $r\in\{0,1,\ldots,a-1\}$ such that
\begin{equation}\label{eq:residue-r}
  n\equiv rb\pmod a.
\end{equation}
For this integer $r$, one has $n\in S \Longleftrightarrow n\ge rb$.
\end{lemma}
\begin{proof}
Because $\gcd(a,b)=1$, multiplication by $b$ is a bijection on the residue
classes modulo $a$.  Hence there is a unique
$r\in\{0,1,\ldots,a-1\}$ satisfying~\cref{eq:residue-r}.

Assume first that $n\in S$.  Then there exist $u,v\in\NN$ such that $n=ua+vb$.
Reducing this equality modulo $a$ and comparing with~\cref{eq:residue-r},
we obtain $vb\equiv rb\pmod a$.
Since $b$ is invertible modulo $a$, it follows that $v\equiv r\pmod a$.
Thus $v=r+ta$ for some $t\in\ZZ$. We claim that $t\ge0$. Indeed, if $t\le-1$, then $v=r+ta\le r-a\le(a-1)-a=-1$,
contradicting $v\in\NN$.  Therefore $t\ge0$, whence $v\ge r$.  Consequently, $n=ua+vb\ge vb\ge rb$.
Conversely, suppose $n\ge rb$.  By~\cref{eq:residue-r}, the difference $n-rb$ is divisible by $a$.  Since this difference is also nonnegative, there exists $q\in\NN$ such that $n-rb=qa$.
Therefore $n=qa+rb\in\gen{a,b}=S$. This proves the reverse implication.
\end{proof}

We now derive a complete membership for $S_a=\gen{a,a+1}$.  This will be used repeatedly.

\begin{lemma}\label{lem:membership}
Let $a\ge2$, and let $n\in\NN$.  Write the Euclidean division of $n$ by $a$
as $n=qa+j$, $q\in\NN$, $0\le j\le a-1$.
Then $n\in\gen{a,a+1} \Longleftrightarrow j\le q$.
\end{lemma}
\begin{proof}
Assume first that $j\le q$.  Then $q-j\in\NN$, and $n=qa+j=(q-j)a+j(a+1)$. Both coefficients $q-j$ and $j$ are nonnegative integers.  Hence $n\in\gen{a,a+1}$.
Conversely, assume that $n\in\gen{a,a+1}$.  Then there exist $x,y\in\NN$ such that $n=xa+y(a+1)=(x+y)a+y$. Write the Euclidean division of $y$ by $a$ as $y=ua+j'$, $u\in\NN$, $0\le j'\le a-1$.
We have $n=(x+y+u)a+j'$. Therefore, we have $q=x+y+u$, $j=j'$.
Since $x,y,u\ge0$ and $y=ua+j\ge j$, one has, more directly, $q=x+y+u\ge y\ge j$.
Thus $j\le q$, proving the converse implication.
\end{proof}

By Lemma~\ref{lem:membership}, we immediately obtain the following result.

\begin{proposition}\label{prop:gaps-consecutive}
Let $a\ge2$. Then
\begin{align*}
 \Gap(\gen{a,a+1})=\left\{qa+j:0\le q\le a-2,\quad q+1\le j\le a-1\right\}.
\end{align*}
\end{proposition}

For later use, we describe the gaps by reflecting the elements of the
semigroup across its Frobenius number.

\begin{proposition}\label{prop:reflection-param}
Let $F_a:=a^2-a-1$. The map
\begin{equation}\label{eq:reflection-map}
\Phi_a:\{(i,j)\in\NN^2:0\le j\le i\le a-2\}\longrightarrow\Gap(\gen{a,a+1}),\qquad\Phi_a(i,j)=F_a-ia-j,
\end{equation}
is a bijection. 
\end{proposition}
\begin{proof}
Let $(i,j)$ satisfy $0\le j\le i\le a-2$.  Then $F_a-ia-j=(a-2-i)a+(a-1-j)$.
Set $q=a-2-i$, $r=a-1-j$.
The inequalities $0\le i\le a-2$ imply $0\le q\le a-2$, while $0\le j\le i$ gives $r-q=1+i-j\ge1$.
Thus $q+1\le r$.  Also $r\le a-1$.  Proposition~\ref{prop:gaps-consecutive} therefore shows that $qa+r=F_a-ia-j$ is a gap.  Hence $\Phi_a$ is well-defined.

Conversely, let $n$ be a gap.  By Proposition~\ref{prop:gaps-consecutive}, it has an Euclidean expression $n=qa+r$, $0\le q\le a-2$, $q+1\le r\le a-1$. Define $i=a-2-q$, $j=a-1-r$. Then $i,j\in\NN$, $i\le a-2$, and $i-j=r-q-1\ge0$.
Thus $0\le j\le i\le a-2$, and a direct calculation gives $F_a-ia-j=qa+r=n$.
This construction is inverse to \cref{eq:reflection-map}; therefore $\Phi_a$ is bijective.
\end{proof}

\begin{definition}[Symmetric numerical semigroup, \cite{Rosales1996} \cite{Rosales2001}]
Let $S\neq\NN$ be a numerical semigroup with Frobenius number $g(S)$. The semigroup $S$ is called \emph{symmetric} if, for every integer $n$ with $0\le n\le g(S)$, exactly one of $n$ and $g(S)-n$ belongs to $S$.
Equivalently,
\[n\notin S\quad\Longleftrightarrow\quad g(S)-n\in S \qquad(0\le n\le g(S)).\]
\end{definition}

\begin{proposition}{\em (\cite[Corollary 4.7]{J.C.Rosales})}\label{prop:symmetry}
Let $a,b\in\ZZ_{>0}$ satisfy $\gcd(a,b)=1$, and put $S=\gen{a,b}$. Then $S$ is symmetric.
\end{proposition}

The next proposition is the basic bridge between the original semigroup and its quotient.  It is valid for every symmetric numerical semigroup, not only for two-generated ones.

\begin{proposition}{\em (\cite[Corollary 4.4]{F.Strazzanti})}\label{prop:quotient-residue}
Let $S\neq\NN$ be a symmetric numerical semigroup with Frobenius number $F=g(S)$.  Let $p\in\ZZ_{>0}$, and assume $S/p\neq\NN$.  Define
\begin{align*}
\omega_p(S):=\min\{s\in S:s\equiv F\pmod p\}.
\end{align*}
Then $\omega_p(S)\le F$, and
\begin{align*}
g\!\left(\frac{S}{p}\right)=\frac{F-\omega_p(S)}{p}.
\end{align*}
\end{proposition}

Applying Proposition~\ref{prop:quotient-residue} to $S_a=\gen{a,a+1}$ and using Proposition~\ref{prop:reflection-param}, we
obtain the following concrete formulation.

\begin{corollary}\label{cor:modular-minimization}
Let $a>p\ge1$, and set $F_a=a^2-a-1$.  Define
\begin{equation}\label{eq:def-w-ap}
   w(a,p):=\min\{s\in\gen{a,a+1}:s\equiv F_a\pmod p\}.
\end{equation}
Then
\begin{equation}\label{eq:G-via-w}
   G(a,p)=\frac{F_a-w(a,p)}{p}.
\end{equation}
Equivalently,
\begin{equation}\label{eq:w-triangle}
 w(a,p)=\min\left\{ia+j:0\le j\le i\le a-2,\quad ia+j\equiv F_a\pmod p\right\}.
\end{equation}
\end{corollary}
\begin{proof}
Proposition~\ref{prop:symmetry} shows that $\gen{a,a+1}$ is symmetric, so Proposition~\ref{prop:quotient-residue}
applies and yields \cref{eq:G-via-w}.
It remains to justify \cref{eq:w-triangle}.  The proof of
Proposition~\ref{prop:reflection-param} shows that every element
$s\in\gen{a,a+1}$ satisfying $0\le s\le F_a$ has an Euclidean form $s=ia+j$, $0\le j\le i\le a-2$.
Proposition~\ref{prop:quotient-residue} gives
$w(a,p)\le F_a-p<F_a$, so the minimum in \cref{eq:def-w-ap} lies in this
range.  Therefore minimizing over the congruence class in the semigroup is
equivalent to the finite minimization in \cref{eq:w-triangle}.
\end{proof}

Before the end of this section, we collect some definitions and results that will be used later.

\begin{definition}\label{def:polynomial-description}
Let $D\subseteq\ZZ_{>0}^2$, and let $H:D\to\CC$ be a function. A \emph{polynomial relation} for $H$ on $D$ is a nonzero polynomial
$F\in\CC[X_1,X_2,X_3]$ such that $F(a,p,H(a,p))=0$ for every $(a,p)\in D$.
A \emph{finite polynomial description} of $H$ on $D$ is a finite collection $f_1,\ldots,f_m\in\CC[X_1,X_2]$ such that, for every $(a,p)\in D$, at least one of the equalities $H(a,p)=f_i(a,p)$ holds.
\end{definition}

A finite polynomial description always produces a polynomial relation, since
\[\prod_{i=1}^m\bigl(Y-f_i(X_1,X_2)\bigr)
\]
is a nonzero polynomial. 
The converse need not hold. Therefore, proving the nonexistence of every polynomial relation is a stronger conclusion than ruling out finitely many polynomial formulas.

\begin{definition}[Zariski density]\label{def:zariski}
A subset $E\subseteq\Aff^n_{\CC}$ is \emph{Zariski dense} if the only polynomial in $\CC[Z_1,\ldots,Z_n]$ that vanishes at every point of $E$ is the zero polynomial.  Equivalently, its Zariski closure is the whole affine
space $\Aff^n_{\CC}$.
\end{definition}

Thus the nonexistence of a polynomial relation for a function of two parameters is exactly the Zariski density of its graph in $\Aff^3_{\CC}$.

We also use the following classical form of Dirichlet's theorem on primes in arithmetic progressions~\cite{Davenport}.
\begin{theorem}[Dirichlet's theorem]\label{thm:Dirichlet}
Let $q\ge2$ and let $r$ be an integer with $\gcd(r,q)=1$. Then the arithmetic progression $r+q\ZZ$ contains
infinitely many prime numbers.
\end{theorem}

\section{Frobenius numbers of certain quotient semigroups}\label{Section-Three-FN}

We now identify a large collection of parameter regions on which the Frobenius number of $\gen{a,b}/p$ has a simple closed form.  This is the key semigroup theoretic construction in the proof of Theorem \ref{thm:main}.

\begin{proposition}\label{prop:explicit-branch}
Let $p>2$ be an odd integer, and let $k$ be an integer satisfying
\begin{equation}\label{eq:k-range}
  2\le k\le\frac{p+1}{2}.
\end{equation}
Let $a,b\in\ZZ_{>0}$ satisfy $p<a<b$, $\gcd(a,b)=1$,
\begin{equation}\label{eq:branch-congruences}
a\equiv1\pmod p, \qquad b\equiv p-k+1\pmod p,
\end{equation}
and
\begin{equation}\label{eq:branch-slope}
  p-k<\frac{b}{a}<p-k+1.
\end{equation}
Then
\begin{equation}\label{eq:explicit-branch}
  g\!\left(\frac{\gen{a,b}}{p}\right)
  =\frac{ab-(k-1)a-2b}{p}.
\end{equation}
\end{proposition}

\begin{proof}
Set $S:=\gen{a,b}$ and $F:=g(S)=ab-a-b$.
Define  $c:=(k-2)a+b$. We shall prove that $\omega_p(S)=c$, where $\omega_p(S)$ is the residue minimum from Proposition~\ref{prop:quotient-residue}.

We first verify that $c$ belongs to the relevant congruence class. 
From \cref{eq:branch-congruences}, we have 
\begin{align*}
F=ab-a-b \equiv(1)(1-k)-1-(1-k)\equiv-1 \pmod p,
\end{align*}
and
\begin{align*}
c=(k-2)a+b \equiv(k-2)+(1-k)\equiv-1 \pmod p.
\end{align*}
Thus $c\equiv F\pmod p$.
Since $c=(k-2)a+b$ with $k-2\ge0$, one also has $c\in S$. Consequently, $\omega_p(S)\le c$.

We now prove the reverse inequality.  Suppose, for contradiction, that there
exists an element $s\in S$ such that
\begin{equation}\label{eq:s-smaller-c}
  s<c
  \qquad\text{and}\qquad
  s\equiv F\pmod p.
\end{equation}
Because $s\in S$, there exist $u,v\in\NN$ with $s=ua+vb$.
Using~\cref{eq:branch-congruences} and
$F\equiv-1\pmod p$, the congruence in~\cref{eq:s-smaller-c} becomes
\begin{equation}\label{eq:u-v-congruence}
  u-(k-1)v\equiv-1\pmod p.
\end{equation}
We separate the proof into three cases according to the value of $v$.

\smallskip
\noindent\emph{Case 1: $v=0$.} \cref{eq:u-v-congruence} gives $u\equiv-1\pmod p$.
Since $u\in\NN$, this implies $u\ge p-1$.  Hence
\begin{equation}\label{eq:case-v0-lower}
  s=ua\ge(p-1)a.
\end{equation}
On the other hand, the upper inequality in~\cref{eq:branch-slope} gives $b<(p-k+1)a$.
Therefore
\begin{align*}
c=(k-2)a+b<(k-2)a+(p-k+1)a=(p-1)a.
\end{align*}
Together with~\cref{eq:case-v0-lower}, this yields $s>c$, contradicting
$s<c$.

\smallskip
\noindent\emph{Case 2: $v=1$.} \cref{eq:u-v-congruence} becomes $u-(k-1)\equiv-1\pmod p$, so $u\equiv k-2\pmod p$.
Because $0\le k-2<p$ and $u\ge0$, it follows that $u\ge k-2$.  Hence $s=ua+b\ge(k-2)a+b=c$, again contradicting $s<c$.

\smallskip
\noindent\emph{Case 3: $v\ge2$.}
Since $u\ge0$, we obtain that $2b\le vb\le ua+vb=s<c=(k-2)a+b$.
Subtracting $b$ gives $b<(k-2)a$.
However, the lower inequality in~\cref{eq:branch-slope} gives $b>(p-k)a$.
The range~\cref{eq:k-range} implies $2k\le p+1$,
and hence $p-k\ge k-1>k-2$.
Therefore, we obtain $b>(p-k)a\ge(k-1)a>(k-2)a$, contradicting $b<(k-2)a$.

The three cases exhaust all possibilities for $v\in\NN$, and each leads to a
contradiction.  Therefore no $s\in S$ satisfying~\cref{eq:s-smaller-c}
exists.  Since $c\in S$ and $c\equiv F\pmod p$, we conclude that $\omega_p(S)=c$.

We must also check that $c<F$, so that the complementary gap used below is
positive.  We have
\begin{align*}
  F-c=(ab-a-b)-((k-2)a+b)=ab-(k-1)a-2b=(a-2)b-(k-1)a.
\end{align*}
Since $b>a$, we have $F-c>(a-2)a-(k-1)a=a(a-k-1)$.
Now~\cref{eq:k-range} and $p\ge3$ imply $k+1\le\frac{p+3}{2}\le p$.
Because $a>p$, it follows that $a-k-1>0$.  Hence $F-c>a(a-k-1)>0$, so indeed $c<F$.

Finally, $p<a<b$ implies $S/p\neq\NN$.  We may therefore apply
Proposition~\ref{prop:quotient-residue}.  Using
$\omega_p(S)=c$, we obtain
\begin{align*}
g(S/p)=\frac{F-c}{p}=\frac{ab-(k-1)a-2b}{p}.
\end{align*}
The numerator is divisible by $p$ because $F\equiv c\pmod p$. This proves~\cref{eq:explicit-branch}.
\end{proof}

\begin{example}\label{ex:branch}
Take $p=5$, $k=3$, $a=11$, $b=28$.
Then $11\equiv1\pmod5$, $28\equiv3=5-3+1\pmod5$, and $5-3=2<\frac{28}{11}<3=5-3+1$.
Proposition~\ref{prop:explicit-branch} gives
\[g\!\left(\frac{\gen{11,28}}{5}\right)=\frac{11\cdot28-2\cdot11-2\cdot28}{5}=46.
\]
Here $g(\gen{11,28})=269$, while the least element of
$\gen{11,28}$ congruent to $269$ modulo $5$ is $\omega_5=(k-2)a+b=11+28=39$.
Thus
\[g\!\left(\frac{\gen{11,28}}{5}\right)=\frac{269-39}{5}=46,
\]
in agreement with the residue-minimum formula.
The above result was also verified using the \emph{numericalsgps} package in GAP \cite{M.Delgado}.
\end{example}

\section{Polynomial nonexistence theorem for $\langle a,b\rangle/p$}\label{Section-Four-TheoI}

\subsection{Arithmetic approximation}

The explicit formula in Proposition~\ref{prop:explicit-branch} is useful only
if its hypotheses are met by sufficiently many coprime pairs $(a,b)$.  The
next lemma supplies such pairs with any prescribed limiting slope in the
open interval appearing in~\cref{eq:branch-slope}.

\begin{lemma}\label{lem:approximation}
Let $p>2$ be a prime, let $2\le k\le\frac{p+1}{2}$, and let $\alpha\in(p-k,p-k+1)$.
Then there exist sequences of positive integers $(a_n)_{n\ge1}$ and
$(b_n)_{n\ge1}$ such that, for every $n$,
\begin{enumerate}
  \item $a_n$ is prime and $a_n\equiv1\pmod p$;
  \item $b_n\equiv p-k+1\pmod p$;
  \item $p<a_n<b_n$, $\gcd(a_n,b_n)=1$, and $p\nmid b_n$;
  \item $p-k<\frac{b_n}{a_n}<p-k+1$;
  \item $\lim_{n\to\infty}\frac{b_n}{a_n}=\alpha$.
\end{enumerate}
In particular, $(a_n,b_n,p)\in\cA$ for every $n$.
\end{lemma}

\begin{proof}
Since $\gcd(1,p)=1$, Dirichlet's theorem \ref{thm:Dirichlet} gives infinitely many primes
congruent to $1$ modulo $p$.  Enumerate an increasing sequence of such primes
as $a_1<a_2<a_3<\cdots$, $a_n\equiv1\pmod p$. 
Then $a_n\to\infty$ as $n\to \infty$.

Set $r:=p-k+1$.
The range on $k$ gives $1\le r\le p-1$,
so in particular $r\not\equiv0\pmod p$.  For each $n$, choose an integer
$m_n$ nearest to the real number $\frac{\alpha a_n-r}{p}$.
Equivalently, choose $m_n\in\ZZ$ so that
\[
  \left|m_n-\frac{\alpha a_n-r}{p}\right|\le\frac12.
\]
Define $b_n:=r+pm_n$. Then $b_n\equiv r=p-k+1\pmod p$,
and multiplication of the preceding approximation inequality by $p$ gives $|b_n-\alpha a_n|\le\frac{p}{2}$.
Therefore, we obtain
\begin{equation}\label{eq:slope-error}
  \left|\frac{b_n}{a_n}-\alpha\right| \le\frac{p}{2a_n}.
\end{equation}
Because $a_n\to\infty$ as $n\to \infty$, the right-hand side tends to zero.  Therefore
\[ \frac{b_n}{a_n}\longrightarrow\alpha \quad \text{as}\quad  n\to \infty.
\]

We now verify the remaining conditions after discarding finitely many initial
terms. Since $\alpha\in(p-k,p-k+1)$, the distance from $\alpha$ to either endpoint is positive.  Set
\[\delta:=\frac12\min\{\alpha-(p-k),\,(p-k+1)-\alpha\}>0. \]
For all sufficiently large $n$, inequality~\cref{eq:slope-error} gives $\left|\frac{b_n}{a_n}-\alpha\right|<\delta$.
It follows that
\[
  p-k<\frac{b_n}{a_n}<p-k+1.
\]
Moreover, the range on $k$ implies $p-k\ge1$, and the inequality is strict;
thus $\alpha>1$.  Hence, for all sufficiently large $n$, $\frac{b_n}{a_n}>1$,
so $b_n>a_n$.  Since $a_n\to\infty$, we may also assume $a_n>p$.  In
particular, $b_n$ is positive.

Next, $b_n\equiv r\not\equiv0\pmod p$, so $p\nmid b_n$ for every $n$.
It remains to prove $\gcd(a_n,b_n)=1$ for all sufficiently large $n$.
Because $a_n$ is prime, a nontrivial common divisor of $a_n$ and $b_n$ would force $a_n\mid b_n$, and hence $b_n/a_n$ would be an integer.  However, the interval $(p-k,p-k+1)$ lies strictly between two consecutive integers and therefore contains no integer.  Once $b_n/a_n$ lies in this interval, it cannot be an integer.  Hence $a_n\nmid b_n$, and therefore $\gcd(a_n,b_n)=1$.

Discarding finitely many initial terms and relabeling the remaining sequence,
we obtain all five asserted properties for every $n$.
\end{proof}

The next elementary lemma is the core of the argument. It says that a polynomial cannot have zeros escaping to infinity along every
slope in a nonempty interval unless it is the zero polynomial.

\begin{lemma}\label{lem:asymptotic-vanishing}
Let $I\subset\mathbb R$ be a nonempty open interval, and let
$G\in\CC[X_1,X_2]$.  Assume that for every $\alpha\in I$ there exists a
sequence $(x_n,y_n)_{n\ge1}$ of real pairs such that
\begin{equation}\label{eq:asymptotic-hypotheses}
  x_n\longrightarrow+\infty,
  \quad
  \frac{y_n}{x_n}\longrightarrow\alpha,\quad \text{as}\quad n\longrightarrow\infty;
  \qquad
  G(x_n,y_n)=0
  \quad\text{for all }n.
\end{equation}
Then $G$ is the zero polynomial.
\end{lemma}

\begin{proof}
Suppose, for contradiction, that $G\neq0$.  Let $D$ be the total degree of
$G$, and write $G$ as the sum of its homogeneous components: $G=G_0+G_1+\cdots+G_D$,
where each $G_j$ is homogeneous of total degree $j$ and $G_D\neq0$.

Fix $\alpha\in I$, and choose a sequence satisfying
\cref{eq:asymptotic-hypotheses}.  Because $x_n\to+\infty$, one has
$x_n\neq0$ for all sufficiently large $n$.  For such $n$, homogeneity gives
\[G_j(x_n,y_n) =x_n^jG_j\!\left(1,\frac{y_n}{x_n}\right).
\]
Therefore
\begin{align*}
 0 =\frac{G(x_n,y_n)}{x_n^D}=G_D\!\left(1,\frac{y_n}{x_n}\right)+\sum_{j=0}^{D-1}x_n^{j-D}G_j\!\left(1,\frac{y_n}{x_n}\right).
\end{align*}
The ratios $y_n/x_n$ converge to $\alpha$, and hence remain bounded.  Each
polynomial function $G_j(1,t)$ is therefore bounded along this sequence.
Since $j-D<0$ for $j<D$, every term in the sum tends to zero.  Passing to the
limit yields $G_D(1,\alpha)=0$.
Because $\alpha\in I$ was arbitrary, the one-variable polynomial $T\longmapsto G_D(1,T)$ vanishes at every point of the nonempty open interval $I$.  A nonzero one-variable complex polynomial has only finitely many zeros, so $G_D(1,T)\equiv0$.
Since $G_D$ is homogeneous of degree $D$, it can be written uniquely as
\[G_D(X_1,X_2) =\sum_{j=0}^{D}c_jX_1^{D-j}X_2^j.
\]
Then $G_D(1,T)=\sum_{j=0}^{D}c_jT^j$.
The identity $G_D(1,T)\equiv0$ forces every coefficient $c_j$ to be zero,
contradicting $G_D\neq0$.  Therefore $G=0$.
\end{proof}

\subsection{Proof of the polynomial nonexistence theorem I}

For completeness, we record the elementary root bound that will be used at the final step.
The following lemma is well known.

\begin{lemma}\label{lem:root-bound}
Let $K$ be a field and let $0\neq H\in K[Y]$ have degree $d$.  Then $H$ has at most $d$ distinct roots in $K$.
\end{lemma}

We next record a simple specialization fact.

\begin{lemma}\label{lem:specialization}
Let $F\in\CC[X_1,X_2,X_3,Y]$ be nonzero.  Then $F(X_1,X_2,q,Y)\in\CC[X_1,X_2,Y]$ is nonzero for all but finitely many $q\in\CC$.
\end{lemma}
\begin{proof}
Regard $F$ as a polynomial in $X_1,X_2,Y$ whose coefficients are polynomials
in $X_3$.  Thus
\[F(X_1,X_2,X_3,Y)=\sum_{(i,j,\ell)\in E}c_{i,j,\ell}(X_3)X_1^iX_2^jY^\ell,
\]
where $E$ is finite and each
$c_{i,j,\ell}(X_3)\in\CC[X_3]$.  Since $F\neq0$, at least one coefficient, say $c_{i_0,j_0,\ell_0}(X_3)$, is a nonzero polynomial.
If the specialization $F(X_1,X_2,q,Y)$ is the zero polynomial, then every coefficient after specialization must vanish.  In particular, $c_{i_0,j_0,\ell_0}(q)=0$.
A nonzero one variable polynomial has only finitely many roots.  Hence only finitely many values of $q$ can make the specialization identically zero.
\end{proof}

\begin{proof}[Proof of Theorem~\ref{thm:main}]
Assume, for contradiction, that a nonzero polynomial $F\in\CC[X_1,X_2,X_3,Y]$ satisfies~\cref{eq:main-relation} for every $(a,b,p)\in\cA$.  Let $d:=\deg_Y F$. By Lemma~\ref{lem:specialization}, the specialization
$F(X_1,X_2,q,Y)$ is nonzero for all but finitely many complex numbers $q$.
There are infinitely many primes, so we may choose an odd prime $q$ such that
$H_q(X_1,X_2,Y) :=F(X_1,X_2,q,Y)\neq 0$.
and $\frac{q-1}{2}>d$.

For each integer
\begin{equation}\label{eq:k-values-main}
  k\in\left\{2,3,\ldots,\frac{q+1}{2}\right\},
\end{equation}
define
\begin{align*}
\Phi_{q,k}(X_1,X_2):=\frac{X_1X_2-(k-1)X_1-2X_2}{q}\in\CC[X_1,X_2].
\end{align*}
We shall prove that
\begin{equation}\label{eq:H-vanishes-branch}
  H_q\bigl(X_1,X_2,\Phi_{q,k}(X_1,X_2)\bigr)\equiv0
\end{equation}
for every $k$ in~\cref{eq:k-values-main}.

Fix such a $k$, and set $G_{q,k}(X_1,X_2):=H_q\bigl(X_1,X_2,\Phi_{q,k}(X_1,X_2)\bigr)$.
Let $I_{q,k}:=(q-k,q-k+1)$.
Fix an arbitrary $\alpha\in I_{q,k}$.  By Lemma~\ref{lem:approximation}, there exist sequences $(a_n)$ and $(b_n)$ such
that
\[(a_n,b_n,q)\in\cA,\qquad\frac{b_n}{a_n}\longrightarrow\alpha\quad \text{as}\quad n\longrightarrow \infty,
\]
and all the hypotheses of Proposition~\ref{prop:explicit-branch} hold with
$p=q$.  Therefore
\begin{align*}
  g\!\left(\frac{\gen{a_n,b_n}}{q}\right) =\Phi_{q,k}(a_n,b_n)
\end{align*}
for every $n$.

Since $(a_n,b_n,q)\in\cA$, the assumed polynomial relation gives
\begin{align*}
0&=F\left(a_n,b_n,q,g\!\left(\frac{\gen{a_n,b_n}}{q}\right)\right)
=H_q\left(a_n,b_n,g\!\left(\frac{\gen{a_n,b_n}}{q}\right)\right)
\\& =H_q\bigl(a_n,b_n,\Phi_{q,k}(a_n,b_n)\bigr)
=G_{q,k}(a_n,b_n).
\end{align*}
Also $a_n\to\infty$ because the $a_n$ are an increasing sequence of primes,
and $b_n/a_n\to\alpha$.  Since this construction works for every
$\alpha\in I_{q,k}$, Lemma~\ref{lem:asymptotic-vanishing} implies $G_{q,k}\equiv0$.
This proves~\cref{eq:H-vanishes-branch}.

Now let $K:=\CC(X_1,X_2)$ be the field of rational functions in $X_1$ and $X_2$.  By
\cref{eq:H-vanishes-branch}, each $\Phi_{q,k}$ is a root of the nonzero
polynomial $H_q\in K[Y]$.  These roots are pairwise distinct.  Indeed, if
$k\neq\ell$, then
\begin{align*}
\Phi_{q,k}-\Phi_{q,\ell}=\frac{-(k-1)X_1+(\ell-1)X_1}{q}=\frac{\ell-k}{q}X_1,
\end{align*}
which is a nonzero element of $K$.

The set~\cref{eq:k-values-main} contains exactly $\frac{q-1}{2}$ integers.  Hence $H_q$, viewed as a nonzero polynomial in $K[Y]$, has at least $(q-1)/2$ distinct roots.  By Lemma~\ref{lem:root-bound},
\begin{equation}\label{eq:degree-lower}
  \deg_Y H_q\ge\frac{q-1}{2}.
\end{equation}
On the other hand, specialization in $X_3$ cannot increase the degree in
$Y$, so
\begin{equation}\label{eq:degree-upper}
  \deg_Y H_q\le\deg_Y F=d.
\end{equation}
Combining~\cref{eq:degree-lower} and~\cref{eq:degree-upper} gives $d\ge\frac{q-1}{2}$, contradicting~$\frac{q-1}{2}>d$.  This contradiction proves that no such nonzero polynomial $F$ exists.
\end{proof}

For each fixed $q$ and $k$, identity~\cref{eq:H-vanishes-branch} says that $Y-\Phi_{q,k}(X_1,X_2)$ divides $H_q$ in
$\CC[X_1,X_2,Y]$, because division by this monic linear polynomial leaves the remainder $H_q(X_1,X_2,\Phi_{q,k})$.  Equivalently, $qY-X_1X_2+(k-1)X_1+2X_2$ is a linear factor after multiplication by the nonzero scalar $q$.
Over $\CC(X_1,X_2)[Y]$, the factors for distinct $k$ are distinct, and their product divides $H_q$. 

\begin{remark}
For a fixed odd prime $p$, Proposition~\ref{prop:explicit-branch} produces the
formulas
\[\frac{ab-(k-1)a-2b}{p},\qquad2\le k\le\frac{p+1}{2}.
\]
There are $(p-1)/2$ such expressions.  Lemma~\ref{lem:approximation} shows
that each expression occurs on infinitely many coprime pairs $(a,b)$.  Consequently, a polynomial relation valid on
all these pairs must admit each expression as a distinct root in the
$Y$-variable.  Since $(p-1)/2$ is unbounded, no relation of fixed $Y$-degree
can exist.
\end{remark}

\begin{proof}[Proof of Corollary~\ref{cor:finite-poly-intro}]
Suppose that finitely many polynomials $f_1,\ldots,f_m\in\CC[X_1,X_2,X_3]$
have the stated property.  Define
\begin{equation}\label{eq:product-polynomial}
P(X_1,X_2,X_3,Y):=\prod_{i=1}^{m}\bigl(Y-f_i(X_1,X_2,X_3)\bigr).
\end{equation}
Every factor in~\cref{eq:product-polynomial} is monic of degree one in $Y$,
so $P$ is a nonzero polynomial, indeed a monic polynomial of degree $m$ in
$Y$.
Fix $(a,b,p)\in\cA$.  By hypothesis, there is an index $i$ such that $f_i(a,b,p) =g\!\left(\frac{\gen{a,b}}{p}\right)$.
At $Y=g\!\left(\frac{\gen{a,b}}{p}\right)$, the $i$th factor of~\cref{eq:product-polynomial} vanishes.  Therefore
\[ P\left(a,b,p,g\!\left(\frac{\gen{a,b}}{p}\right)\right)=0\]
for every $(a,b,p)\in\cA$.  This contradicts Theorem~\ref{thm:main}.

Finally, suppose a finite family worked for all relatively prime positive integers $a<b$ and all $p<a$.  It would then work on the subset $\cA$, and the same contradiction would follow.
\end{proof}

\begin{corollary}\label{cor:zariski}
The set
\[\left\{\left(a,b,p,g\!\left(\frac{\gen{a,b}}{p}\right)\right):(a,b,p)\in\cA\right\}
\]
is Zariski dense in affine four-space $\mathbb A^4_{\CC}$.
\end{corollary}
\begin{proof}
A subset of affine space is Zariski dense precisely when the only polynomial vanishing on the entire subset is the zero polynomial.  The assertion is therefore an equivalent reformulation of Theorem~\ref{thm:main}.
\end{proof}

The same argument excludes finite collections of rational functions.

\begin{corollary}\label{cor:rational}
There do not exist finitely many rational functions $R_1,\ldots,R_m\in\CC(X_1,X_2,X_3)$ such that, for every $(a,b,p)\in\cA$, at least one $R_i$ is defined at $(a,b,p)$ and satisfies
\[R_i(a,b,p)=g\!\left(\frac{\gen{a,b}}{p}\right).
\]
\end{corollary}
\begin{proof}
Write $R_i=\frac{A_i}{B_i}$, $A_i,B_i\in\CC[X_1,X_2,X_3]$, $B_i\neq0$. Define
\[P(X_1,X_2,X_3,Y):=\prod_{i=1}^{m}\bigl(B_i(X_1,X_2,X_3)Y-A_i(X_1,X_2,X_3)\bigr).
\]
Each factor is nonzero because $B_i\neq0$, and the polynomial ring over
$\CC$ is an integral domain.  Hence $P\neq0$.
For a given $(a,b,p)\in\cA$, choose an index $i$ for which $R_i$ is defined
and equals the Frobenius number.  Then $B_i(a,b,p)\neq0$ and $B_i(a,b,p)g\!\left(\frac{\gen{a,b}}{p}\right)-A_i(a,b,p)=0$.
Thus one factor of $P$ vanishes, so $P\left(a,b,p,g\!\left(\frac{\gen{a,b}}{p}\right)\right)=0$.
This contradicts Theorem~\ref{thm:main}.
\end{proof}

Theorem~\ref{thm:main} is an algebraic nonexistence theorem.  It does not
assert that the Frobenius number of $\gen{a,b}/p$ is uncomputable, nor does it
exclude formulas involving floor functions, modular inverses, minimization,
continued fractions, or a number of cases that grows with $p$.  What it
excludes is a single nonzero polynomial relation of bounded degree, and hence
any finite list of polynomial or rational expressions that is independent of
$p$.

\section{Exact residue class formula for $\langle a,a+1\rangle/p$ and fixed $p$}\label{sec:exact}

After discussing the formula for the Frobenius number $g\!\left(\frac{\gen{a,b}}{p}\right)$, it is natural to consider the formula for the Frobenius number $g\!\left(\frac{\gen{a,a+1}}{p}\right)$ under the stronger condition $b=a+1$.

The quotient of the numerical semigroup $\gen{a,a+1}$ is an important object of study.
In \cite{Rosales03}, a numerical semigroup $S(a,b,c)$ is called \emph{proportional modular} if
$S(a,b,c)=\{x\in \mathbb{N} \mid ax \mod b \leq cx\}$ and $a,b,c \in \mathbb{P}$. 
It is shown \cite{AMRobles08} that a numerical semigroup is proportional modular if and only if it is the quotient of $\gen{a,a+1}$.

\subsection{Quasi-polynomial of period $p$}

We now evaluate the minimum in \cref{eq:w-triangle}.  The result is valid
for every integer $p\ge2$; primality is not needed in this section.
For $p\ge1$ and $u\in\ZZ$, let $\remp_p(u)$ denote the unique integer $v\in\{0,1,\ldots,p-1\}$ such that
$v\equiv u\pmod p$.

Fix $p\ge2$ and $r\in\{0,1,\ldots,p-1\}$.  For $i\in\{0,1,\ldots,p-1\}$, set
\begin{align*}
\delta_{p,r}(i):=\remp_p\bigl(r^2-r-1-ir\bigr).
\end{align*}

\begin{lemma}\label{lem:stopping-index}
For every $p\ge2$ and every $r\in\{0,1,\ldots,p-1\}$, the set $\{i\in\{0,1,\ldots,p-1\}:\delta_{p,r}(i)\le i\}$
is nonempty.
\end{lemma}
\begin{proof}
We can check that $i=p-1$ belongs to this set.
\end{proof}

By Lemma~\ref{lem:stopping-index}, the following definition is meaningful.

\begin{definition}\label{def:IJ}
For $p\ge2$ and $r\in\{0,1,\ldots,p-1\}$, define
\begin{align}
 I_p(r)&:=\min\{i\in\{0,1,\ldots,p-1\}:\delta_{p,r}(i)\le i\},\label{eq:def-I}\\
 J_p(r)&:=\delta_{p,r}\bigl(I_p(r)\bigr).\nonumber
\end{align}
Then we have 
\begin{equation}\label{eq:IJ-bounds}
   0\le J_p(r)\le I_p(r)\le p-1.
\end{equation}
\end{definition}

\begin{theorem}\label{thm:exact-formula}
Let $a>p\ge2$, let $r=\remp_p(a)$. Then $w(a,p)=I_p(r)a+J_p(r)$, and hence
\begin{equation}\label{eq:G-exact}
G(a,p)=\frac{a^2-(I_p(r)+1)a-(J_p(r)+1)}{p}.
\end{equation}
\end{theorem}

\begin{proof}
We divide the proof into four steps.

\smallskip
\noindent\emph{Firstly: the candidate $I_p(r)a+J_p(r)$ belongs to the semigroup.}
By \cref{eq:IJ-bounds}, $0\le J_p(r)\le I_p(r)$.  Therefore $I_p(r)a+J_p(r)=(I_p(r)-J_p(r))a+J_p(r)(a+1)$.
Both coefficients on the right hand side are nonnegative integers.  Thus $I_p(r)a+J_p(r)\in\gen{a,a+1}$.

\smallskip
\noindent\emph{Secondly: the candidate lies in the required congruence class.}
Since $a\equiv r\pmod p$, one has
\begin{equation}\label{eq:F-residue}
   F_a=a^2-a-1\equiv r^2-r-1\pmod p.
\end{equation}
By the definition of $J_p(r)$, $J_p(r)=\remp_p(r^2-r-1-I_p(r)r)$, so $I_p(r)r+J_p(r)\equiv r^2-r-1\pmod p$.
Replacing $r$ by the congruent integer $a$ gives $I_p(r)a+J_p(r)\equiv F_a\pmod p$.
Thus $I_p(r)a+J_p(r)$ is an admissible element in the minimum defining $w(a,p)$.

\smallskip
\noindent\emph{Thirdly: the candidate is smaller than every other admissible semigroup element.}
Let $s\in\gen{a,a+1}$ satisfy
\begin{equation}\label{eq:s-admissible}
   s\equiv F_a\pmod p.
\end{equation}
Write the Euclidean division of $s$ by $a$ as $s=qa+v$, $q\in\NN$, $0\le v\le a-1$.
Since $s\in\gen{a,a+1}$, Lemma~\ref{lem:membership} gives $v\le q$.
We compare $q$ with $I_p(r)$.

Suppose first that $q<I_p(r)$.  Since $I_p(r)\le p-1$, one has $q\le I_p(r)-1\le p-2$.
Together with $v\le q$, this implies $0\le v\le q<p$.
Reducing \cref{eq:s-admissible} modulo $p$ and using $a\equiv r\pmod p$
and \cref{eq:F-residue}, we obtain $qr+v\equiv r^2-r-1\pmod p$, or equivalently, $v\equiv r^2-r-1-qr\pmod p$.
Because $v\in\{0,1,\ldots,p-1\}$, uniqueness of the least nonnegative
residue yields $v=\delta_{p,r}(q)$.
But $v\le q$, so $\delta_{p,r}(q)\le q$.  This contradicts the minimality of
$I_p(r)$ in \cref{eq:def-I}, because $q<I_p(r)$.  Hence the case $q<I_p(r)$ is impossible.

Suppose next that $q=I_p(r)$.  Then$v\le q$ and $I_p(r)\le p-1$ imply
$0\le v\le p-1$.  The same congruence calculation gives $v=\delta_{p,r}(I_p(r))=J_p(r)$. Therefore $s=I_p(r)a+J_p(r)$.

Finally, suppose that $q>I_p(r)$.  Since $q$ and $I_p(r)$ are integers, $q\ge I_p(r)+1$.
Thus $s=qa+v\ge(I_p(r)+1)a$.
On the other hand, $J_p(r)\le p-1$ and $p<a$ give
$$I_p(r)a+J_p(r)\le I_p(r)a+(p-1)<I_p(r)a+a=(I_p(r)+1)a.$$
Therefore, we obtain $s>I_p(r)a+J_p(r)$.

We have proved that every admissible $s$ is either equal to $I_p(r)a+J_p(r)$ or is
strictly larger.  Therefore $w(a,p)=I_p(r)a+J_p(r)$.

\smallskip
\noindent\emph{Fourth: evaluation of the Frobenius number.}
By Corollary~\ref{cor:modular-minimization}, substituting $F_a=a^2-a-1$ and $w(a,p)=I_p(r)a+J_p(r)$ gives
\begin{align*}
   G(a,p)=\frac{a^2-a-1-I_p(r)a-J_p(r)}{p}=\frac{a^2-(I_p(r)+1)a-(J_p(r)+1)}{p}.
\end{align*}
This is \cref{eq:G-exact}.
\end{proof}

The numerator in \cref{eq:G-exact} is divisible by $p$ because
\[ a^2-(I_p(r)+1)a-(J_p(r)+1)=F_a-(I_p(r)a+J_p(r))\equiv0\pmod p.\]

\begin{definition}\label{def:quasipolynomial}
A function $h:\ZZ_{>0}\to\CC$ is a \emph{quasi-polynomial of period $p$} on a specified range if there exist polynomials
$Q_0,\ldots,Q_{p-1}\in\CC[X]$ such that
\[h(n)=Q_r(n)\quad\text{whenever}\quad n\equiv r\pmod p
\]
throughout that range. The actual minimal period may be a proper divisor of $p$.
\end{definition}

\begin{corollary}\label{cor:quasipolynomial}
Fix $p\ge2$.  For $r\in\{0,\ldots,p-1\}$, define
\begin{align*}
P_{p,r}(X):=\frac{X^2-(I_p(r)+1)X-(J_p(r)+1)}{p}.
\end{align*}
Then, for every integer $a>p$ with $a\equiv r\pmod p$, $G(a,p)=P_{p,r}(a)$.
Thus $a\mapsto G(a,p)$ is a quadratic quasi-polynomial of period dividing $p$ on the range $a>p$.
\end{corollary}
\begin{proof}
This is an immediate restatement of Theorem~\ref{thm:exact-formula}, since $I_p(r)$ and $J_p(r)$ depend only on $p$ and the residue class $r$.
\end{proof}

\begin{remark}[An algorithm]\label{rem:algorithm}
Theorem~\ref{thm:exact-formula} gives the following terminating procedure for computing $G(a,p)$ when $a>p\ge2$:
\begin{enumerate}
\item compute $r=\remp_p(a)$;
\item for $i=0,1,\ldots,p-1$, compute $\delta_{p,r}(i)=\remp_p(r^2-r-1-ir)$;
\item stop at the first $i$ for which $\delta_{p,r}(i)\le i$;
\item set $I_p(r)=i$, $J_p(r)=\delta_{p,r}(i)$, and evaluate \cref{eq:G-exact}.
\end{enumerate}
Lemma~\ref{lem:stopping-index} proves termination no later than $i=p-1$.
This algorithm is a formula involving a finite minimization, but it is not a finite collection of polynomial or rational expressions independent of $p$.
\end{remark}

\begin{example}\label{ex:p=5}
For $p=5$, direct application of Definition~\ref{def:IJ} gives
\[
\begin{array}{c|cc|c}
 r & I_5(r) & J_5(r) & P_{5,r}(X)\\ \hline
 0 & 4 & 4 & \dfrac{X^2-5X-5}{5}\\[4pt]
 1 & 2 & 2 & \dfrac{X^2-3X-3}{5}\\[4pt]
 2 & 2 & 2 & \dfrac{X^2-3X-3}{5}\\[4pt]
 3 & 0 & 0 & \dfrac{X^2-X-1}{5}\\[4pt]
 4 & 4 & 0 & \dfrac{X^2-5X-1}{5}
\end{array}
\]
Thus four distinct quadratic branches occur.  For example, if $a=9$, then
$a\equiv4\pmod5$, and $G(9,5)=\frac{9^2-5\cdot9-1}{5}=7$. Equivalently, the largest integer $n$ for which $5n\notin\gen{9,10}$ is $n=7$.
\end{example}

\subsection{Further discussion for fixed $p$}\label{sec:branches}

For fixed $p$, different residue classes may yield the same pair
$(I_p(r),J_p(r))$.  We now show that, when $p$ is prime, no pair can arise
from more than two residue classes.  

For $p\ge2$, define
\begin{align*}
 \cB_p&:=\{(I_p(r),J_p(r)):r=0,1,\ldots,p-1\},\\
 \cB_p^{\times}&:=\{(I_p(r),J_p(r)):r=1,2,\ldots,p-1\}.
\end{align*}
The second set records only the nonzero residue classes modulo $p$.
For $(i,j)\in\cB_p$, define the associated polynomial
$$P_{p;i,j}(X):=\frac{X^2-(i+1)X-(j+1)}{p}.$$
If $(i,j)=(I_p(r),J_p(r))$, then $P_{p;i,j}(X)=P_{p,r}(X)$.

\begin{lemma}\label{lem:fiber-two}
Let $p$ be prime and let $(i,j)\in\cB_p$.  Then there are at most two
residues $r\in\{0,1,\ldots,p-1\}$ satisfying $(I_p(r),J_p(r))=(i,j)$.
The same bound holds after restricting to nonzero residues.
\end{lemma}
\begin{proof}
Suppose $(I_p(r),J_p(r))=(i,j)$.  By the definition of $J_p(r)$, $j\equiv r^2-r-1-ir\pmod p$.
Rearranging gives $r^2-(i+1)r-(j+1)\equiv0\pmod p$.
Thus every residue $r$ is a root in the field $\FF_p$ of the polynomial
\[
   Z^2-(i+1)Z-(j+1)\in\FF_p[Z].
\]
This is a nonzero polynomial of degree two.  A nonzero degree-two polynomial over a field has at most two roots.
This completes the proof.
\end{proof}

By Lemma \ref{lem:fiber-two}, we obtain the following result.
\begin{proposition}\label{prop:branch-count}
Let $p$ be prime. Then we have $|\cB_p|\ge\left\lceil\frac{p}{2}\right\rceil$.
If $p$ is odd, then $|\cB_p^{\times}|\ge\frac{p-1}{2}$.
\end{proposition}

\begin{lemma}\label{lem:distinct-polynomials}
Fix $p\ge2$. If $P_{p;i,j}(X)=P_{p;i',j'}(X)$ as polynomials in $\CC[X]$, then $(i,j)=(i',j')$.
\end{lemma}
\begin{proof}
Multiplying the asserted equality by $p$ gives $-(i+1)X-(j+1)=-(i'+1)X-(j'+1)$.
Equality of the coefficients of $X$ yields $i=i'$, and equality of the constant terms then yields $j=j'$.
\end{proof}

For fixed $p$, define
\begin{align*}
 \Psi_p(X,Y):=\prod_{(i,j)\in\cB_p}\bigl(pY-X^2+(i+1)X+(j+1)\bigr).
\end{align*}
This polynomial has $Y$-degree $|\cB_p|$.

\begin{theorem}\label{thm:fixed-p-degree}
Fix $p\ge2$. We have: (i): The nonzero polynomial $\Psi_p$ satisfies $\Psi_p(a,G(a,p))=0$ for every integer $a>p$.
(ii): If $H\in\CC[X,Y]$ is nonzero and $H(a,G(a,p))=0$ for every integer $a>p$, then $\deg_Y H\ge|\cB_p|$.
Consequently, $|\cB_p|$ is the minimum possible $Y$-degree.
\end{theorem}
\begin{proof}
The part (i) is an immediate consequence. For part (ii), fix $(i,j)\in\cB_p$.  By definition of $\cB_p$, there exists
$r\in\{0,1,\ldots,p-1\}$ such that $(I_p(r),J_p(r))=(i,j)$.
There are infinitely many integers $a>p$ satisfying $a\equiv r\pmod p$.
For every such $a$, Corollary~\ref{cor:quasipolynomial} gives $G(a,p)=P_{p;i,j}(a)$.
Therefore the one-variable polynomial
\[h_{i,j}(X):=H\bigl(X,P_{p;i,j}(X)\bigr)\in\CC[X]
\]
vanishes at infinitely many integers $a$.  A nonzero polynomial over $\CC$
has only finitely many roots, so
\begin{equation}\label{eq:h-identically-zero}
   H\bigl(X,P_{p;i,j}(X)\bigr)=0
\end{equation}
identically in $\CC[X]$.

Now regard $H$ as a polynomial in $Y$ over the field $\CC(X)$. \cref{eq:h-identically-zero} says that $P_{p;i,j}(X)\in\CC(X)$ is a root of this polynomial.  As $(i,j)$ ranges over $\cB_p$, these roots are distinct by
Lemma~\ref{lem:distinct-polynomials}.  A nonzero polynomial of degree
$\deg_YH$ over a field has at most $\deg_YH$ distinct roots.  Hence $\deg_YH\ge|\cB_p|$.  Part (i) supplies a relation of exactly
that degree, so the lower bound is sharp.
\end{proof}

Combining Theorem~\ref{thm:fixed-p-degree} with Proposition~\ref{prop:branch-count} gives the following.

\begin{corollary}\label{cor:unbounded-fixed-p}
If $p$ is prime and $H\in\CC[X,Y]$ is a nonzero polynomial satisfying $H(a,G(a,p))=0$ for every integer $a>p$, then $\deg_YH\ge\left\lceil\frac{p}{2}\right\rceil$.
Thus the minimum possible $Y$-degree of a fixed-$p$ relation tends to infinity with $p$ through the primes.
\end{corollary}

We next record the corresponding result when $a$ is also restricted to primes.  

\begin{theorem}\label{thm:fixed-p-prime-a}
Let $p>2$ be prime, and let $H\in\CC[X,Y]$ be nonzero.  Suppose
\begin{equation}\label{eq:H-prime-a}
   H(a,G(a,p))=0
\end{equation}
for every prime $a>p$. Then $\deg_YH\ge|\cB_p^{\times}|\ge\frac{p-1}{2}$.
\end{theorem}
\begin{proof}
Fix $(i,j)\in\cB_p^{\times}$.  By definition, there exists a nonzero residue
$r\in\{1,\ldots,p-1\}$ such that $(I_p(r),J_p(r))=(i,j)$.
Because $p$ is prime and $r\ne0$, one has $\gcd(r,p)=1$.  Dirichlet's theorem
implies that there are infinitely many primes $a$ satisfying
$a\equiv r\pmod p$.  All but finitely many of these primes are greater than
$p$.

For each such prime $a$, Theorem~\ref{thm:exact-formula} gives $G(a,p)=P_{p;i,j}(a)$.
By the assumption \cref{eq:H-prime-a}, the polynomial $H\bigl(X,P_{p;i,j}(X)\bigr)\in\CC[X]$
vanishes at infinitely many prime values of $X$, and therefore is identically
zero.  Hence every polynomial indexed by $\cB_p^{\times}$ is a root of
$H$ when $H$ is regarded as an element of $\CC(X)[Y]$.  These roots are
distinct by Lemma~\ref{lem:distinct-polynomials}.  Consequently, $\deg_YH\ge|\cB_p^{\times}|$. This completes the proof.
\end{proof}

\section{Polynomial nonexistence theorem for $\langle a,a+1\rangle/p$}\label{sec:global}

\subsection{Proof of the polynomial nonexistence theorem II}

We now prove that no fixed polynomial relation can persist as $p$ varies.
The key point is that a nonzero polynomial in the three variables $(X,T,Y)$
has a fixed $Y$-degree, whereas Theorems~\ref{thm:fixed-p-degree} and
\ref{thm:fixed-p-prime-a} force the required $Y$-degree after specialization
$T=p$ to grow linearly with $p$.

The following lemma is similar to Lemma \ref{lem:specialization}; we omit its proof.

\begin{lemma}\label{lem:specialization2}
Let $F\in\CC[X,T,Y]$ be nonzero.  For $t\in\CC$, define $F_t(X,Y):=F(X,t,Y)\in\CC[X,Y]$.
Then $F_t$ is the zero polynomial for only finitely many values of $t$.
Moreover, $\deg_YF_t\le\deg_YF$ for every $t$.
\end{lemma}

\begin{theorem}\label{thm:global-all-a}
There is no nonzero polynomial $F\in\CC[X,T,Y]$ satisfying $F(a,p,G(a,p))=0$ for every $(a,p)\in\cD_{\mathrm{all}}$, where
$\cD_{\mathrm{all}}$ is defined in \cref{eq:Dall}.
\end{theorem}
\begin{proof}
Suppose, for contradiction, that such a nonzero polynomial $F$ exists.  Let $d:=\deg_YF$. By Lemma~\ref{lem:specialization2}, only finitely many complex numbers $t$ have the property that $F_t(X,Y)=F(X,t,Y)$ is identically zero.  Since there are
arbitrarily large primes, we may choose a prime $p$ such that $F_p\ne0$ and $\left\lceil\frac{p}{2}\right\rceil>d$.

For every integer $a>p$, the pair $(a,p)$ belongs to $\cD_{\mathrm{all}}$.  Hence the assumed relation gives $F_p(a,G(a,p))=F(a,p,G(a,p))=0$. The polynomial $F_p\in\CC[X,Y]$ is nonzero.  Corollary
\ref{cor:unbounded-fixed-p} therefore implies $\deg_YF_p\ge\left\lceil\frac{p}{2}\right\rceil$.
On the other hand, Lemma~\ref{lem:specialization2} gives $\deg_YF_p\le\deg_YF=d$.
We have 
\[d\ge\deg_YF_p \ge\left\lceil\frac{p}{2}\right\rceil>d,
\]
an impossibility. Therefore no such nonzero $F$ exists.
\end{proof}

\begin{corollary}\label{cor:zariski-all-a}
The set $\Gamma_{\mathrm{all}}:=\{(a,p,G(a,p)):(a,p)\in\cD_{\mathrm{all}}\}\subseteq\Aff^3_{\CC}$ is Zariski dense in $\Aff^3_{\CC}$.
\end{corollary}
\begin{proof}
By Definition~\ref{def:zariski}, failure of Zariski density would mean that
some nonzero polynomial in $\CC[X,T,Y]$ vanishes at every point of
$\Gamma_{\mathrm{all}}$.  This is excluded by
Theorem~\ref{thm:global-all-a}.
\end{proof}

We now prove the sparse version for $(a,p)\in\cD_{\mathrm{pp}}$, which is closest in spirit to Curtis's theorem.

\begin{theorem}\label{thm:global-prime-prime}
There is no nonzero polynomial $F\in\CC[X,T,Y]$ satisfying $F(a,p,G(a,p))=0$ for every $(a,p)\in\cD_{\mathrm{pp}}$, where $\cD_{\mathrm{pp}}$ is defined in \cref{eq:Dpp}.
\end{theorem}
\begin{proof}
Assume, toward a contradiction, that a nonzero polynomial $F$ satisfying $F(a,p,G(a,p))=0$ exists.  Let $d=\deg_YF$.
By Lemma~\ref{lem:specialization2}, the specialization
$F_p(X,Y)=F(X,p,Y)$ is nonzero for all but finitely many $p\in\CC$.
Choose an odd prime $p$ such that $F_p\ne0$ and $\frac{p-1}{2}>d$.
Such a prime exists because the set of exceptional $p$ is finite and there
are arbitrarily large primes.

For every prime $a>p$, the pair $(a,p)$ lies in $\cD_{\mathrm{pp}}$, so $F_p(a,G(a,p))=F(a,p,G(a,p))=0$.
Theorem~\ref{thm:fixed-p-prime-a}, applied to the nonzero polynomial $F_p$,
gives $\deg_YF_p\ge\frac{p-1}{2}$.
Lemma~\ref{lem:specialization2} gives the opposite upper bound $\deg_YF_p\le\deg_YF=d$.
We obtain $d\ge\deg_YF_p\ge\frac{p-1}{2}>d$, which is impossible.  Hence no such $F$ exists.
\end{proof}

\begin{corollary}\label{cor:zariski-prime-prime}
The set $\Gamma_{\mathrm{pp}}:=\{(a,p,G(a,p)):(a,p)\in\cD_{\mathrm{pp}}\} \subseteq\Aff^3_{\CC}$ is Zariski dense in $\Aff^3_{\CC}$.
\end{corollary}
\begin{proof}
This is equivalent to Theorem~\ref{thm:global-prime-prime} by Definition~\ref{def:zariski}.
\end{proof}

For each fixed $p$, Theorem~\ref{thm:fixed-p-degree} constructs the explicit nonzero relation $\Psi_p(X,Y)=0$.  Thus it would be false to claim that $a\mapsto G(a,p)$ satisfies no polynomial relation for fixed $p$.
The global theorem is possible because the minimum required $Y$-degree is at least $\lceil p/2\rceil$ and therefore is unbounded as $p$ varies. 
The contrast is: for fixed $p$, finite quadratic quasi-polynomial description exists; for variable $p$, no polynomial relation of bounded degree exists.

\subsection{Zariski closure and corollaries}\label{sec:four-variable}

Let $\mathbb{K}$ be an algebraically closed field and let $\mathbb{A}^n_\mathbb{K}$ denote the affine $n$-space over $\mathbb{K}$. We set $\mathbb{K}[x_1, \dots, x_n]$ for its coordinate ring.
For a subset $\mathcal{S} \subseteq \mathbb{K}[x_1, \dots, x_n]$, the \emph{zero locus} of $\mathcal{S}$ is defined as
\[V(\mathcal{S}) := \left\{ P \in \mathbb{A}^n_\mathbb{K} \mid f(P) = 0 \ \text{for all } f \in \mathcal{S} \right\}.
\]
Conversely, for a subset $\mathcal{T} \subseteq \mathbb{A}^n_\mathbb{K}$, the \emph{vanishing ideal} of $\mathcal{T}$ is defined as
\[I(\mathcal{T}) := \left\{ f \in \mathbb{K}[x_1, \dots, x_n] \mid f(P) = 0 \ \text{for all } P \in \mathcal{T} \right\}.
\]
In particular, for any $\mathcal{T} \subseteq \mathbb{A}^n_\mathbb{K}$, its Zariski closure is given by $\overline{\mathcal{T}} = V(I(\mathcal{T}))$.

\begin{theorem}\label{thm:four-variable-ideal}
Let $\widetilde\Gamma_{\mathrm{pp}}=\left\{\left(a,a+1,p,G(a,p)\right):(a,p)\in\cD_{\mathrm{pp}}\right\}\subseteq\Aff^4_{\CC}$.
Then
\begin{equation}\label{eq:four-variable-ideal}
 I\bigl(\widetilde\Gamma_{\mathrm{pp}}\bigr)=\bigl(X_2-X_1-1\bigr).
\end{equation}
Equivalently, the Zariski closure of $\widetilde\Gamma_{\mathrm{pp}}$ is the hyperplane $V(X_2-X_1-1)\subseteq\Aff^4_{\CC}$.
\end{theorem}

\begin{proof}
Every point of $\widetilde\Gamma_{\mathrm{pp}}$ has second coordinate equal
to its first coordinate plus one.  Therefore $X_2-X_1-1$ vanishes on the set, and hence
\begin{equation}\label{eq:easy-ideal-inclusion}
 (X_2-X_1-1)\subseteq I(\widetilde\Gamma_{\mathrm{pp}}).
\end{equation}

For the reverse inclusion, let $H(X_1,X_2,T,Y)\in I(\widetilde\Gamma_{\mathrm{pp}})$.
Define its restriction to the hyperplane by
\begin{align*}
   \widehat H(X,T,Y):=H(X,X+1,T,Y)\in\CC[X,T,Y].
\end{align*}
For every $(a,p)\in\cD_{\mathrm{pp}}$, we have $\widehat H(a,p,G(a,p))=H(a,a+1,p,G(a,p))=0$.
By Theorem~\ref{thm:global-prime-prime}, the only polynomial in
$\CC[X,T,Y]$ with this property is the zero polynomial.  Hence
\begin{equation}\label{eq:Hhat-zero}
   \widehat H(X,T,Y)=0.
\end{equation}

We now show that \cref{eq:Hhat-zero} is equivalent to divisibility by
$X_2-X_1-1$.  Regard $H$ as a polynomial in the single variable $X_2$ with
coefficients in the ring $R:=\CC[X_1,T,Y]$.
Since $X_2-(X_1+1)=X_2-X_1-1$ is monic in $X_2$, the division algorithm in $R[X_2]$ gives unique
polynomials $Q\in R[X_2]$ and $R_0\in R$ such that
\begin{equation}\label{eq:division-H}
   H=(X_2-X_1-1)Q+R_0,
\end{equation}
where the remainder $R_0$ is independent of $X_2$.  Substituting
$X_2=X_1+1$ into \cref{eq:division-H} yields $H(X_1,X_1+1,T,Y)=R_0(X_1,T,Y)$.
The left-hand side is zero by \cref{eq:Hhat-zero}.  Therefore $R_0=0$, and
\cref{eq:division-H} shows that $H\in(X_2-X_1-1)$.
Thus $I(\widetilde\Gamma_{\mathrm{pp}})\subseteq(X_2-X_1-1)$.
Combining this with \cref{eq:easy-ideal-inclusion} proves \cref{eq:four-variable-ideal}.

The statement about the Zariski closure follows because the vanishing ideal of the hyperplane $V(X_2-X_1-1)$ is exactly the principal ideal generated by $X_2-X_1-1$.
\end{proof}

\begin{corollary}\label{cor:no-finite-poly-reduced}
There do not exist polynomials
$f_1,\ldots,f_m\in\CC[X,T]$ such that, for every $(a,p)\in\cD_{\mathrm{pp}}$, at least one index $i$ satisfies $G(a,p)=f_i(a,p)$.
The same conclusion holds with $\cD_{\mathrm{pp}}$ replaced by $\cD_{\mathrm{all}}$ or by the full set of integer pairs $a>p\ge2$.
\end{corollary}
\begin{proof}
Suppose such polynomials exist. Define $F(X,T,Y):=\prod_{i=1}^m\bigl(Y-f_i(X,T)\bigr)$.
Each factor is monic of degree one in $Y$, so their product is a nonzero
polynomial in $\CC[X,T,Y]$.  For every $(a,p)\in\cD_{\mathrm{pp}}$, at least
one equality in $G(a,p)=f_i(a,p)$ holds.  The corresponding
factor in $F(X,T,Y)$ is zero at $(a,p,G(a,p))$.  Hence $F(a,p,G(a,p))=0$ for every $(a,p)\in\cD_{\mathrm{pp}}$.  This contradicts Theorem~\ref{thm:global-prime-prime}.
This completes the proof.
\end{proof}

\begin{corollary}\label{cor:no-finite-rational}
There do not exist finitely many rational functions
\[R_i(X,T)=\frac{A_i(X,T)}{B_i(X,T)}\in\CC(X,T),\qquad B_i\ne0,
\]
with the property that, for every $(a,p)\in\cD_{\mathrm{pp}}$, at least one
$R_i$ is defined at $(a,p)$ and satisfies $G(a,p)=R_i(a,p)$.
\end{corollary}
\begin{proof}
Suppose such rational functions exist.  Define $F(X,T,Y):=\prod_{i=1}^m\bigl(B_i(X,T)Y-A_i(X,T)\bigr)$.
Every factor is nonzero because $B_i\ne0$, and $\CC[X,T,Y]$ is an integral domain.  Hence $F$ is nonzero.
At any $(a,p)\in\cD_{\mathrm{pp}}$, choose an index $i$ for which
$B_i(a,p)\ne0$ and $G(a,p)=\frac{A_i(a,p)}{B_i(a,p)}$.
Then $B_i(a,p)G(a,p)-A_i(a,p)=0$, so the corresponding factor in $F(X,T,Y)$ vanishes.
Therefore $F(a,p,G(a,p))=0$ for every $(a,p)\in\cD_{\mathrm{pp}}$, contradicting
Theorem~\ref{thm:global-prime-prime}.
\end{proof}

Similarly, the results do not say that $G(a,p)$ is uncomputable or that no expression using floor functions, least residues, minimization, or Euclidean algorithms can exist.  The exact formula in Theorem~\ref{thm:exact-formula} is already an effective formula of that broader kind.  What is ruled out is a single algebraic relation of bounded degree, and therefore any finite family
of polynomial or rational branches independent of $p$.

\section{Concluding remarks}\label{Section-Finily-CR}

Curtis's polynomial nonexistence theorem \cite{F.Curtis} for $\gen{a,b,c}$ constructs infinitely many semigroups on which the Frobenius number is linear in two of the generators, then uses points at infinity and a degree argument to rule out a global polynomial relation.  In the quotient setting, the symmetry of $\gen{a,b}$ first reduces the problem to the least element of a prescribed congruence class.  The resulting formula is bilinear in $a$ and $b$.

The discussion of the non-existence theorems for the polynomial formulas of the Frobenius numbers $g\!\left(\frac{\gen{a,b}}{p}\right)$ and $g\!\left(\frac{\gen{a,a+1}}{p}\right)$ differs.
The specialization $b=a+1$ changes both the geometry and the arithmetic of the Frobenius number of a quotient of a numerical semigroup.

First, the parameter space loses one dimension.  Consequently, the original four-variable nonexistence statement cannot survive verbatim: the graph is contained in the hyperplane $b-a-1=0$.  Theorem \ref{thm:four-variable-ideal} proves that the graph is nevertheless Zariski dense inside that hyperplane. 

Second, the local structure becomes more explicit.  For each fixed $p$, the
function $a\mapsto G(a,p)$ is not merely computable but is a quadratic
quasi-polynomial, with determined by the finite residue procedure
in Definition~\ref{def:IJ}.  This is stronger local information than one has
for a general pair $(a,b)$.

Third, this local simplification does not produce a finite algebraic description as $p$ varies. Even under the restriction that both 
$a$ and $p$ are prime, we prove that there is no polynomial formula for the Frobenius number $g\!\left(\frac{\gen{a,a+1}}{p}\right)$.

Nevertheless, finding closed-form formulas for the Frobenius numbers of quotients of certain special numerical semigroups remains an interesting research direction.






\noindent
{\small \textbf{Acknowledgments:}}
The author would also like to express sincere gratitude for all the suggestions that have improved the presentation of this paper.

\end{document}